\documentclass[12pt,leqno]{amsart}
\usepackage{latexsym,amssymb,amsmath,multicol, rotating,lscape}
\usepackage{letltxmacro}

\usepackage{tikz}
\usetikzlibrary{calc,decorations.markings}

 2

\newtheorem{thm}{Theorem}[section]
\newtheorem{lem}[thm]{Lemma}
\newtheorem{prop}[thm]{Proposition}

\newcommand{\thmref}[1]{Theorem~\ref{#1}}
\newcommand{\lemref}[1]{Lemma~\ref{#1}}

\newcommand{\propref}[1]{Proposition~\ref{#1}}

\theoremstyle{remark}
\newtheorem{rmk}{Remark}[section]

\begin{document}

\title[Moment of Hecke eigenvalues]
{First moment of Hecke eigenvalues at the integers represented by binary quadratic forms}

\author{Manish Kumar Pandey and Lalit Vaishya}
\address{(Manish Kumar Pandey ) Department of Mathematics,  UPES, Dehradun-248007, Uttarakhand, India}
\email{gilbardopandey@gmail.com, manishkumar.p@srmap.edu.in}
\address{( Lalit Vaishya) School of Mathematics, The Institute of Mathematical Sciences, A CI of Homi Bhabha
National Institute, CIT Campus, Taramani, Chennai 600 113, India,}
\email{lalitvaishya@gmail.com, lalitv@imsc.res.in}


\date{\today}
 
\maketitle

\begin{abstract}

In the article, we consider a question concerning the estimation of summatory function of the Fourier coefficients of Hecke eigenforms indexed by a sparse set of integers. In particular, we provide an estimate for the following sum; 
\begin{equation*}
\begin{split}
S(f, \mathcal{Q}; X ) &:=  \sideset{}{^{\flat }}\sum_{n= \mathcal{Q}(\underline{x}) \le X \atop \gcd(n,N) =1  } \lambda_{f}(n),
\end{split}
\end{equation*}
where $\flat$ means that sum runs over the square-free positive integers, $\lambda_{f}(n)$ denotes the normalised $n^{\rm th}$ Fourier coefficients of a Hecke eigenform $f$ of integral weight $k$ for the congruence subgroup $\Gamma_{0}(N)$ and $\mathcal{Q}$ is a primitive integral positive-definite binary quadratic forms of fixed discriminant $D<0$ with the class number $h(D)=1$.  As a consequence, we determine the size, in terms of conductor of associated $L$-function, for the first sign change of Hecke eigenvalues indexed by the integers which are represented by $\mathcal{Q}$. This work is an improvement and generalisation of the previous results. 


\end{abstract}

\section{Introduction and statements of the results}

Let $M_{k}(\Gamma_{0}(N), \chi)$ and $S_{k}(\Gamma_{0}(N), \chi)$ denote the ${\mathbb C}$- vector space of modular forms and cusp forms respectively of integral weight $k$ for the congruence subgroup $\Gamma_{0}(N)$ with nebentypus $\chi$. Moreover, if the  nebentypus is trivial, we denote the space of modular forms and  cusp forms by $M_{k}(\Gamma_{0}(N))$ and $S_{k}(\Gamma_{0}(N))$, respectively. A cusp form $f \in S_{k}(\Gamma_{0}(N))$ is said to be a Hecke eigenform (newform) if $f$ is a common eigenform for all the Hecke operators and the Atkin-Lehner $W$-operators. The space of   Hecke eigenforms (newforms) is denoted by $S^{new}_{k}(\Gamma_{0}(N))$. The Fourier coefficients of a Hecke eigenform satisfy the famous Ramanujan-Petersson bound which follows from the work of Deligne \cite{Deligne}. With these notations, let $f \in S_{k}(\Gamma_{0}(N))$ be a Hecke eigenform with the Fourier series expansion at the cusp $\infty$ given by:
\begin{equation}\label{Hol-Four Exp}
f(\tau)=\sum_{n=1}^\infty \lambda_f(n)n^{\frac{k-1}{2}}q^{n},
\end{equation}
where $q:=e^{2\pi i \tau}$ and $\tau \in \mathbb{H}$. $f$ is said to be normalised if $\lambda_f(1)=1$. The Fourier coefficients $\lambda_f(n)$ are known as the normalised $n^{th}$-Fourier coefficients of $f$. The normalised Fourier coefficient $\lambda_{f}(n)$ is a multiplicative function and satisfies the Hecke relation \cite[Eq. (6.83)]{Iwaniec}:
\begin{equation}\label{HeckeR}
\lambda_{f}(m)\lambda_{f}(n) = \sum_{d \vert \gcd(m,n)}  \lambda_{f}\left(\frac{mn}{d^2}\right),  
\end{equation}
 for all positive integers $m$ and $n$ coprime to $N$.  From Deligne's estimate of Fourier coefficients of cusp forms, we have 
\begin{equation} \label{lambda-coefficient-bound}
|\lambda_{f}(n)| \le \tau(n) \ll_{\epsilon} n^{\epsilon}, 
\end{equation}
for  any arbitrary small  $\epsilon > 0,$ where $\tau(n)$ denotes the number of positive divisors of $n$.

The arithmetic behaviour and distribution of values of arithmetical functions is the center of attraction in analytic number theory. The randomness in the behaviour of these arithmetical functions leads to many interesting results. In this context, the arithmetical function \linebreak $n \mapsto \lambda_{f}(n)$ (where $\lambda_{f}(n)$ is the $n^{\rm th}$ Fourier coefficients of a Hecke eigenform $f$) is one of the interesting host. The problems become more fascinating and harder if one considers the arithmetic behaviour and distribution over the sparse set of positive integers. Iwaniec and Kowalski \cite{HIwaniec} studied the  distribution of $\{\lambda_{f}(p): p- {\rm prime}\}$ and proved that there exist a positive constant $c$ such that
$$
\displaystyle{\sum_{p\le X}} \lambda_{f}(p) \ll_{f} X \exp \left(-c \sqrt{\log X}\right),
$$
which is an analogue of the prime number theorem for Hecke eigenforms. In this regard, a most popularly known problem is studied by Blomer \cite{blomer}, where he considers the distribution of the sequence $\{ \lambda_{f}(q(n)): n \in \mathbb{N}\},$ where $q(x)$ is a monic quadratic polynomial. More precisely, he proves that 
$$
\displaystyle{\sum_{n\le X}} \lambda_{f}(q(n)) \ll X^{\frac{6}{7}+\epsilon},
$$
for any $\epsilon>0.$ For a polynomial with more than one variable, the problem has been studied   broadly for many arithmetic functions such as Von Mangoldt function ($n \mapsto \Lambda(n)$) and generalised divisor functions ($n \mapsto \tau_{k}(n)$) etc. In his remarkable work, Acharya \cite{Ratna} has studied a two variable analog of the sum studied in the work of Blomer \cite{blomer}, where he studied  the distribution of $\{\lambda_{f}(q(a,b))\}$ where $q(x,y) = x^2+y^2$ and obtained an estimate for the summatory function  $\displaystyle{\sum_{k,l \in \mathbb{Z} \atop k^2+l^2 \le X}} \lambda_{f}(q(k,l)).$  For details, we refer to \cite{Ratna}.

In this article, our motivation is to generalize the result of Acharya \cite{Ratna} by studying the distribution of $\{\lambda_{f}(\mathcal{Q}(\underline{x}))\}$ where $\mathcal{Q}$ is a primitive integral positive-definite binary quadratic forms of fixed discriminant $D<0$ with the class number $h(D)=1.$ In previous works, the second author study the higher power moment $\displaystyle{\sideset{}{^{\flat }}\sum_{n= \mathcal{Q}(\underline{x}) \le X \atop \gcd(n,N) =1 } } \lambda_{f}^{j}(n)$ for $j \le 8$ in \cite{Lalit, LalitV}, and estimates for more general higher power moment has been carried in \cite{Manish-Lalit}. Before stating our result, we fix some notations.

\bigskip
Throughout the paper, Let $\mathcal{Q}(\underline{x})$ be a primitive integral positive-definite binary quadratic (reduced) form given by; $\mathcal{Q}(\underline{x}) = ax_1^2 + b x_1x_2 + c x_2^2,$ where $\underline{x} = (x_1,x_2)\in \mathbb{Z}^2$, $a, b, c \in \mathbb{Z}$ with $\gcd(a, b, c) = 1$ and fixed discriminant $D = b^2-4ac < 0$. Further,  we assume that for such discriminant $D$, the class number  $h(D)$ is $ 1$. For details, we refer to \cite[chapter 2]{Cox}. 

We consider the following sum for an upper-bound estimate:
\begin{equation}\label{Sum-PIBQF}
\begin{split}
S(f, \mathcal{Q}; X )  &:=  \sideset{}{^{\flat }}\sum_{n= \mathcal{Q}(\underline{x}) \le X \atop \gcd(n,N) =1  } \lambda_{f}(n),
\end{split}
\end{equation}
where $\lambda_{f}(n)$ denotes the $n^{\rm th}$ Fourier coefficient of a Hecke eigenform $f$ and symbol $\flat $ means that the sum runs over all square-free integers.
We also produce the explicit dependency in terms of weight $k$ and level $N$ of the Hecke eigenform $f$, and discriminant $D$ of $\mathcal{Q}.$ More precisely, we obtain the following estimate.

\begin{thm}\label{PropUpper}
Let  $f \in S_{k}(\Gamma_{0}(N))$ be a normalised Hecke eigenform and $\mathcal{Q}$ be a reduced form of discriminant $D$ with the class number $h(D)=1$. For sufficiently large $X>0$ and any arbitrary small $\epsilon>0,$ we have  
\begin{equation*}
\begin{split}
S(f, \mathcal{Q}; X )  &=  \sideset{}{^{\flat }}\sum_{n= \mathcal{Q}(\underline{x}) \le X \atop \gcd(n,N) =1  } \lambda_{f}(n) \ll_{\epsilon} {(Nk^{2}|D|)}^{\frac{1}{2}+\epsilon} X^{\frac{1}{2}+\epsilon}.\\
\end{split}
\end{equation*}
where the implied constant is absolute.
\end{thm}

\begin{rmk}
If the class  number  $h(D) >1$, the same estimate (obtained in \thmref{PropUpper}) holds for the sum 
\begin{equation}\label{Sum-PIBQF}
\begin{split}
S(f, D; X )  &:=  \sideset{}{^{\flat }}\sum_{\substack{n= \mathcal{Q}(\underline{x}) \le X \\  ~\mathcal{Q} \in  \mathcal{S}_{D}, \underline{x} \in \mathbb{Z}^{2} \\ \gcd(n,N) =1 } } \lambda_{f}(n),
\end{split}
\end{equation}
where $\mathcal{S}_{D}$ denotes the set of in-equivalent primitive integral positive-definite binary quadratic (reduced) forms of fixed discriminant with the class number $h(D) = \#\mathcal{S}_{D}$.
\end{rmk}

A consequence of the above theorem is a result on the sign change of the Fourier coefficients. Below, we give a brief description of the sign change problem and state our second result. It is well-known that if the Fourier coefficients of a cusp form are real, it changes its sign infinitely often. The sequence of Fourier coefficients of Hecke eigenform always has a sign change in the interval of length $X,$  for sufficiently large $X.$ The problem of getting first ever sign change of $\lambda_{f}(n), $ is quite hard.  As mentioned in \cite{Koh-Sen}, it is a $GL_{2}$-analogue of the $GL_{1}$ problem which concerns about finding least quadratic non-residue.  The problem is quite hard in the sense that $\lambda_{f}(n)$ may take infinitely many values and it can be very small. We also determine the first sign change of the real sequence of Hecke eigenvalues over the integers which are represented by some function $\mathcal{U}$. In particular, if   $\mathcal{U}$ is an identity function on $\mathbb{N}$ or characteristic function on the set of primes, the problem of first sign changes and counting the number of sign changes in short interval $[X, X+h],$ for sufficiently large $X$ and some $h<X,$ has been considered by many authors. 
Let $\mathcal{I}$ denote the identity function on $\mathbb{N}$  and $n_{f, \mathcal{I}}$ denotes the first sign change of Fourier coefficients $\lambda_{f}(n)$ of Hecke eignforms $f \in S_{k}(\Gamma_{0}(N)),$ i.e., the least integer among all $n \in\mathbb{N}$ such that $\lambda_{f}(n)<0$ and $\gcd(n, N) =1$. 
The problem of bounding  $n_{f,\mathcal{I}}$ in terms of analytic conductor has attracted many mathematicians. Kohnen and Sengupta \cite{Koh-Sen} have proved that the bound of $n_{f,\mathcal{I}}$ is given by 
$$
n_{f,\mathcal{I}} \ll kN(\log k)^{A} \exp\left( c_{2} \sqrt{\frac{\log(N+1)}{\log\log(N+2)}} \right),
$$
for some constant $A, c_{2}>0$, using analytic properties of automorphic $L$-functions.  Jointly with Iwaniec \cite{Iwa-Koh-Sen}, they improved the bound for $n_{f, \mathcal{I}}$ and obtained that $n_{f,\mathcal{I}} \ll Q^{29/60},$ where $Q$ is the analytic conductor, of the $L$-function associated to $f$, given explicitly in terms of weight and level of the form $f.$ Kowalski et al \cite{KLSW2010} has made an important observation that the Hecke relation \eqref{HeckeR} allows one to replace $\lambda_{f}(n)$ by a step function when one wants to find the lower bound of partial sum of $\lambda_{f}(n).$ This allowed to improve the exponent, i.e., $n_{f,\mathcal{I}} \ll Q^{9/20}.$  In the article  \cite{KLSW2010}, they mentioned that the exponent $9/20$ is not the limit of the method. This was further improved by Matom$\ddot{\rm a}$ki \cite{KMatomaki} to $3/8$ by deep analysis of summatory function of $\lambda_{f}(n)$ over square-free integers and this is best known bound for $n_{f, \mathcal{I}}$. The current known bound for $n_{f,\mathcal{I}}$  is still far from the truth as it is known that $n_{f, \mathcal{I}} \ll (\log kN)^{2}$ under generalised Riemann hypothesis.
The problem of getting the first-ever sign change becomes interesting if the sequence of integers is indexed by some function other than the identity function on $\mathbb{N}$. In particular, a result of this sort has been obtained by the first author in joint work with P. Tiwari \cite{Pra-Man}, where the sequence of integers is written as a sum of squares, i.e. by considering the function $Q: \mathbb{Z}^2 \rightarrow \mathbb{N}$ given by $Q({x, y}) = x^{2}+y^{2}. $ Following the argument of Kohnen and Sengupta \cite{Koh-Sen}, Pandey and Tiwari \cite{Pra-Man} have proved that 
$$n_{f, Q} \le k^{8}N^{4} \log^{64}(kN) ~ \max \left( \!\! N^{4}  \exp\left( \!\! c{\frac{\log(N+1)}{\log\log(N+2)}}  \right)(\log\log(N+2))^{4} \psi^{4}_{k}(N), \log^{120}(kN) \!\! \right),$$
where $\psi_{k}(N) = \displaystyle{\prod_{p \mid N} \frac{\log(kN)}{\log p}}.$

The method of the first author \cite{Pra-Man} can be generalised to any primitive integral positive-definite binary quadratic forms of fixed discriminant with the class number $1.$ In this work adopting the method of Kowalski et al. and Matomaki, we improve the previous result of the first author \cite{Pra-Man} and generalise the result to any primitive integral positive-definite binary quadratic form of fixed discriminant with the class number $1$.

Now, we state our second result concerning the first-ever sign change of Hecke eigenvalues of a normalised Hecke eigenform at the integers represented by $\mathcal{Q}(\underline{x})$. Let $n_{f, \mathcal{Q}}$ be the least integer among all $n \in \mathbb{N}$ such that  $\lambda_{f}(n) < 0,$ with $\gcd(n, N)=1$ and $n$ is represented by $\mathcal{Q}(\underline{x})$.
With these notations, we observe that $n_{f,\mathcal{I}} \neq n_{f, \mathcal{Q}}$ as it is not necessary that $n_{f,\mathcal{I}}$ can always be represented by $\mathcal{Q}(\underline{x})$. Moreover, we observe that $n_{f,\mathcal{I}} \le n_{f, \mathcal{Q}}.$ Hence, an upper bound estimate for $ n_{f, \mathcal{Q}}$ will also work for $n_{f,\mathcal{I}}$.  An estimate for $n_{f, \mathcal{Q}}$ is given by the following result.

\begin{thm}\label{ExtMatKLSW}
Let $k \ge 2$ be an integer and $N$ be a square-free positive integer. Let \linebreak $f \in S_{k}(\Gamma_{0}(N))$ be a Hecke eigenform and $\mathcal{Q}$ be a primitive integral positive-definite binary quadratic form of fixed discriminant $D$ with the class number $1.$. Then, for any arbitrarily small $\epsilon>0$, 
$$
n_{f, \mathcal{Q}} \ll_{\epsilon} (Nk^{2} |D|^{2})^{\frac{3}{4} + \epsilon}
$$
where the implied constant is absolute.
\end{thm}

\begin{rmk} 
 For the primitive integral positive-definite binary quadratic form of the following fundamental discriminants $D = -3, -4, -8, -7, -11, -19, -43, -67, -163$ and non-fundamental discriminants $D =  -12, -16, -27, -28,$ the class number $h(D) = 1$. The list of reduced forms corresponding to these discriminants is given in  \cite[Page 19- 20]{Buell}. \thmref{ExtMatKLSW} holds only for primitive integral positive-definite binary quadratic form of such discriminants.
\end{rmk}

\begin{rmk} 
If the class number $h(D) >1$ for the discriminant $D<0$, then there are $h(D)$ many reduced forms corresponding to discriminants $D$. Moreover,  let $n_{f, D}$ be the least integer among all $n \in \mathbb{N}$ such that  $\lambda_{f}(n)< 0,$ with $\gcd(n, N)=1$ and $n$ is represented by some of them from $h(D)$ many reduced forms corresponding to discriminants $D$. Then, 
\begin{equation}
\begin{split}
n_{f, D} \ll_{\epsilon} (Nk^{2} |D|^{2})^{\frac{3}{4} + \epsilon} h(D)^{-3/2}.
\end{split}
\end{equation}
\end{rmk}

\smallskip
\textbf{Acknowledgement :} We would like to thank Prof. B. Ramakrishnan for encouraging us to work on this problem and for his guidance and suggestions while preparing the manuscript. We also thank R. Acharya for his helpful suggestions. The second author would like to thank IMSc, Chennai for its hospitality and financial support through an institute fellowship.

\section{Key Ingredients}
In order to obtain our first result, we need to study the following sum:
\begin{equation*}
\begin{split}
S(f, \mathcal{Q}; X ) &=  \sideset{}{^{\flat }}\sum_{n= \mathcal{Q}(\underline{x}) \le X \atop \gcd(n,N) =1  } \lambda_{f}(n). 
\end{split}
\end{equation*}
where symbol $\flat $ means that the sum runs over all square-free integers. The function $r_{\mathcal{Q}}(n)$ being always non-negative allows us to write the above
partial sums defined in terms of known arithmetical functions modulo certain weight, i.e.,
\begin{equation}
\begin{split}
S(f, \mathcal{Q}; X ) 
& =  \sideset{}{^{\flat }} \sum_{n \le X \atop \gcd(n,N) =1 } \lambda_{f}(n) r_{\mathcal{Q}}(n) = w_{D} \sideset{}{^{\flat }} \sum_{n \le X \atop \gcd(n,N) =1 } \lambda_{f}(n) r^{*}_{\mathcal{Q}}(n),
\end{split}
\end{equation}
 where $r_{\mathcal{Q}}(n)$ denotes the number of representations of $n$ by the quadratic form $\mathcal{Q}(\underline{x}) $, $w_{D}$ is size of group of units and $r_{\mathcal{Q}}(n) = w_{D} r^{*}_{\mathcal{Q}}(n)$. We define the generating function  $\theta_{\mathcal{Q}}(\tau)$ for $r_{\mathcal{Q}}(n)$  associated to binary quadratic  form $\mathcal{Q}(\underline{x})$ as follows:  
\begin{equation*}
\begin{split}
\theta_{\mathcal{Q}}(\tau) & := \displaystyle{\sum_{\underline{x} \in {\mathbb Z}^{2} } q^{\mathcal{Q}(\underline{x})}}  \quad = \quad  \displaystyle{\sum_{n = 0}^{\infty} r_{\mathcal{Q}}(n) q^{n}}, \qquad   \quad ~~~ q = e^{2 \pi i \tau}. \\
\end{split}
\end{equation*} 
The generating function $\theta_{\mathcal{Q}}(\tau) \in M_{1}(\Gamma_{0}(|D|), \chi_{D}) $ (see \cite[Theorem 10.9]{Iwaniec}), where $\chi_{D}$ is the Dirichlet character modulo $|D|$ which  is given by Jacobi symbol $\chi_{D}(d) := \left(\frac{D}{d} \right)$. It is well-known that  $r_{\mathcal{Q}}(n) \ll n^{\epsilon}$.  

We are considering the quadratic forms $\mathcal{Q}(\underline{x})$ of discriminant $D < 0$ such that the class number $h(D)$ is $1$ and there is a well-known formula for $r_{\mathcal{Q}}(n)$ given by \cite[section 11.2]{Iwaniec}:
\begin{equation} \label{r-quad}
r_{\mathcal{Q}}(n) = w_{D} \sum_{d \vert n} \chi_{D}(d), \   \
{\rm ~where~} \ \ \ \ \ 
w_{D} =
\begin{cases} 
6 {\rm ~~~~~ if~~~~~~} D = -3,\\
4 {\rm ~~~~~ if~~~~~~} D = -4,\\
2 {\rm ~~~~~ if~~~~~~} D < -4. \end{cases} 
\end{equation}
 The formula for $r_{\mathcal{Q}}(n)$ depends only on the  discriminant $D$ but not on the choice of $a, b$ and $c$ appearing in the definition of binary quadratic from $\mathcal{Q}(\underline{x}).$ For details, we refer to \cite[Section 2]{Lalit, LalitV}.

In order to obtain an upper bound for the sum $S(f, \mathcal{Q}; X),$ We consider the following Dirichlet series: 
\begin{equation}\label{R-L-function}
\begin{split}
L(f, \mathcal{Q}; s ) &=  \sideset{}{^{\flat }} \sum_{n \ge 1 \atop \gcd(n,N) =1 } \frac{\lambda_{f}(n) r^{*}_{\mathcal{Q}}(n)}{n^{s}}, 
\end{split}
\end{equation}
where $r^{*}_{\mathcal{Q}}(n) = \displaystyle{\sum_{d \vert n}}~ \chi_{D}(d).$ The Dirichlet series $L(f, \mathcal{Q}; s )$ converges absolutely and uniformly for $\Re(s) >1.$   To obtain an upper bound for the sum $S(f, \mathcal{Q}; X ),$ we first decompose $L(f, \mathcal{Q}; s )$ in terms of known Hecke $L$-functions. Before stating the decomposition of $L(f, \mathcal{Q}; s )$ in terms of well-known function, we define the $L$-functions associated to a Hecke eigenform.  The Hecke $L$-function  associated to a normalised Hecke eigenform $f(\tau) = \displaystyle{\sum_{n=1}^\infty \lambda_f(n)n^{\frac{k-1}{2}}q^{n}} \in S_{k}(\Gamma_{0}(N))$  is given by:  
\begin{equation}
\begin{split}
L(s, f) &= \sum_{n \ge 1} \frac{\lambda_{f}(n)}{n^{s}} = \prod_{p \mid N} \left(1-\frac{\lambda_{f}(p)}{p^s}\right)^{-1}  \prod_{p \nmid N}\left(1-\frac{\lambda_{f}(p)}{p^{s} }+ \frac{1}{p^{2s}}\right)^{-1}, \qquad \Re(s) >1.
\end{split}
\end{equation}
The completed Hecke $L$-function is defined as follows:
\begin{equation}
\begin{split}
\Lambda(s,f) := \left(\frac{\sqrt{N}}{2\pi}\right)^{s} \Gamma(s+k-1) ~~ L(s, f), 
\end{split}
\end{equation}
which is analytically continued to the whole complex plane and satisfies a nice functional equation $(s \rightarrow 1-s).$ 
For a given Dirichlet character $\chi$ of modulus $m,$ the twist of $f$ is defined as
$f \otimes \chi(\tau) := \displaystyle{\sum_{n=1}^\infty \lambda_f(n) \chi(n) n^{\frac{k-1}{2}}q^{n}} \in S_{k}(\Gamma_{0}(M), \chi^{2}),$ where $M \mid Nm^{2}.$ 
The twisted Hecke $L$- function associated to $f \otimes \chi$ is given as follows:
\begin{equation}
\begin{split}
L(s, f \otimes \chi ) &= \sum_{n \ge 1} \frac{\lambda_{f}(n) \chi(n)}{n^{s}} = \!\! \prod_{p \mid M} \!\!   \left(1-\frac{\lambda_{f}(p) \chi(p)}{p^s}\right)^{-1}  \!\!\! \prod_{p \nmid M}   \left(1-\frac{\lambda_{f}(p)\chi(p)}{p^{s} }+ \frac{1}{p^{2s}}\right)^{-1} \!\!. \\ 
\end{split}
\end{equation}
This also converges absolutely, uniformly and non-vanishing for $\Re(s) >1.$ The completed twisted Hecke $L$-function is defined as follows:
\begin{equation}
\begin{split}
\Lambda(s,f \otimes \chi) := \left(\frac{\sqrt{M}}{2\pi}\right)^{s} \Gamma(s+k-1)L(s, f\otimes \chi), 
\end{split}
\end{equation}
which also has analytic continuation to whole $\mathbb{C}$-plane and satisfies a nice functional equation $(s \rightarrow 1-s).$ For details, see \cite[Section 7.2]{Iwaniec}. The decomposition of  $L(f, \mathcal{Q}; s )$ is given by the following Lemma.
\begin{lem}\label{Decomp}
For $\Re(s) >1 $, we have 
\begin{equation} \label{U(s)}
\begin{split}
L(f, \mathcal{Q}; s ) & = L(s, f) L(s, f \otimes {\chi_{D}}) G(s), \\ 
\end{split}
\end{equation}
where $L(s, f)$ and  $L(s, f \otimes {\chi_{D}})$ are Hecke $L$-function and twisted Hecke $L$-function, respectively and the Euler product for $G(s)$ is given by 
\begin{equation*}
\begin{split}
 &  \prod_{p \mid N}\left( 1- \frac{ \lambda_{f}(p)}{p^{s}}\right)  \prod_{p \mid N}\left( 1- \frac{ \lambda_{f}(p) \chi_{D}(p)}{p^{s}}\right)  \prod_{p \nmid N \atop  p \mid |D|}\left( 1- \frac{ \lambda_{f}(p) \chi_{D}(p)}{p^{s}} + \frac{1}{p^{2s}}\right)^{-1}
\\
 &   \prod_{p \nmid N}  \left( 
 \begin{cases}
  1+ \frac{2 - 2 \lambda^{2}_{f}(p)- \lambda^{2}_{f}(p) \chi_{D}(P)  }{p^{2s} } + \frac{ \lambda_{f}(p)(1+\chi_{D}(p)) (1+ \lambda^{2}_{f}(p) \chi_{D}(p)) }{p^{3s} } \\
 \qquad  \qquad  +  \frac{ 1- 2\lambda^{2}_{f}(p)(1+ \chi_{D}(P)) }{p^{4s} } + \frac{\lambda_{f}(p)(1+ \chi_{D}(P))}{p^{5s}}
 \end{cases}
 \right).
\end{split}
\end{equation*}
 It converges absolutely and uniformly in the half plane $\Re(s) > \frac{1}{2}$ and non-zero for $\Re(s) = 1$.
\end{lem}

\begin{proof}
We know that $\lambda_{f}(n)$ and  $r^{*}_{\mathcal{Q}}(n)$ are multiplicative function. So, for $\Re(s) >1,$ we have 
\begin{equation*}
\begin{split}
L(f, \mathcal{Q}; s ) &=  \!\!\!\! \sideset{}{^{\flat }} \sum_{n \ge 1 \atop \gcd(n,N) =1 } \frac{\lambda_{f}(n) r^{*}_{\mathcal{Q}}(n)}{n^{s}} = \prod_{p \nmid N}\left( 1+ \frac{\mu^{2}(p) \lambda_{f}(p) r^{*}_{\mathcal{Q}}(p) }{p^{s}}\right) \\
& =  \prod_{p \nmid N}\left( 1+ \frac{ \lambda_{f}(p) (1+ \chi_{D}(p)) }{p^{s}}\right)= L(s, f) L(s, f \otimes \chi_{D} )G(s), \\
\end{split}
\end{equation*}
where the Euler product for $G(s)$ is given in lemma.
\end{proof}

\begin{lem} \cite[Chapter 5]{HIwaniec}
Let $\epsilon>0$ be an arbitrarily small real number. The hybrid convexity bound of Hecke  $L$-function and twisted Hecke $L$-function is given by:
\begin{equation}
\begin{split}
L(s, f) &\ll_{\epsilon}  \left( N k^{2} (|s|+3)^{2}\right)^{\frac{1}{4}+\epsilon} ~ \\ 
{\rm and} \quad L(s, f \otimes {\chi_{D}}) & \ll_{\epsilon} \left( N  k^{2}  |D|^{2} (|s|+3)^{2}\right)^{\frac{1}{4}+\epsilon} ~ \\ 
\end{split}
\end{equation}
on the line $\Re(s) = \frac{1}{2}+\epsilon.$ The second integral moment is given by 
\begin{equation}
\begin{split}
\int_{T}^{2T} \left|L \left(\frac{1}{2} +\epsilon+it, f\right)\right|^{2} dt & \ll_{\epsilon} (Nk^{2}T^{2})^{\frac{1}{2} +\epsilon}\\
\int_{T}^{2T} \left|L \left(\frac{1}{2} +\epsilon+it, f \otimes \chi_{D}\right)\right|^{2} dt & \ll_{\epsilon} (Nk^{2} |D|^{2}T^{2})^{\frac{1}{2} +\epsilon}
\end{split}
\end{equation}
uniformly for any $|T| \ge 1$. 
\end{lem}

\subsection{ \textbf{Mean value of a multiplicative function over the integers represented by a reduced form of discriminant $D$ with $h(D) =1$}}

Let $\eta$ be a fixed positive integer. We define the following sum given by 
\begin{equation*}
\begin{split}
\mathcal{E}_{\eta}( X ) &=  \sideset{}{^{\flat }}\sum_{n= \mathcal{Q}(\underline{x}) \le X \atop \gcd(n,N) =1  } \eta^{\omega(n)}, 
\end{split}
\end{equation*}   
where $\omega(n)$ denotes the number of distinct prime divisors of a positive integer $n$. The above sum can be expressed in terms of known arithmetical function given by 
\begin{equation}\label{MFE}
\begin{split}
\mathcal{E}_{\eta}( X ) &=  \sideset{}{^{\flat }}\sum_{n= \mathcal{Q}(\underline{x}) \le X \atop \gcd(n,N) =1  } \eta^{\omega(n)} = \sum_{n \le X \atop \gcd(n,N) =1  } \mu^{2}(n) \eta^{\omega(n)} r^{*}_{\mathcal{Q}}(n). \\
\end{split}
\end{equation}  
In order to obtain an estimate for $\mathcal{E}_{\eta}( X )$ we first decompose the associated Dirichlet series (say $D(s)$) in terms of known $L$-functions and then apply the Perron's formula and use the standard argument. The associated Dirichlet series $D(s)$ is given by 
\begin{equation*}
\begin{split}
D(s) &=  \sum_{n \ge 1 \atop \gcd(n,N) =1 } \frac{\mu^{2}(n) \eta^{\omega(n)} r^{*}_{\mathcal{Q}}(n)}{n^{s}}. 
\end{split}
\end{equation*}
The Dirichlet series $D(s)$ converges for $\Re(s) > 1$.  Following the argument as in the proof of  \lemref{Decomp}, one has the decomposition of $D(s)$ as follows.
\begin{equation}\label{DSLF}
\begin{split}
D(s) &= \prod_{p \nmid N }\left(1+ \frac{ \mu^{2}(p)\eta^{\omega(p)} r^{*}_{\mathcal{Q}}(p)}{p^{s}}\right) = \prod_{p \nmid N }\left(1+ \frac{\eta (1+\chi_{D}(p))}{p^{s}}\right) = \zeta^{\eta}(s) L^{\eta}(s, \chi_{D})P(s), \\
\end{split}
\end{equation} 
where $\zeta(s)$ is the Riemann zeta function, and $ L(s, \chi_{D})$ is the Dirichlet $L$- function associated to the quadratic character $\chi_{D}$ and $P(s)$ is given in terms of Euler product as follows:
\begin{equation}\label{PSS}
\begin{split}
P(s)  & =  \prod_{p \mid N} \left[\left( 1- \frac{1}{p^{s}}\right) \left( 1- \frac{\chi_{D}(p)}{p^{s}}\right)\right]^{\eta}  \prod_{p \nmid N} \left [ \left( 1- \frac{1}{p^{s}}\right)^{\eta}\left( 1- \frac{\chi_{D}(p)}{p^{s}}\right)^{\eta}  
\left(1+ \frac{\eta (1+\chi_{D}(p))}{p^{s}}\right) \right].\\   
\\
\end{split}
\end{equation}
The last product converges absolutely and uniformly for $\Re(s) > 1/2$ and non-zero at $s=1.$ Using the Perron's formula, an estimate for $\mathcal{E}_{\eta}( X )$ is given in the following proposition.
\begin{prop}  \label{MVEBQF}
Let $\mathcal{E}_{\eta}( X )$ be as in \eqref{MFE}. Then there exist an absolute constant $C = C(\eta)$ such that 
\begin{equation}\label{MFEST}
\begin{split}
\mathcal{E}_{\eta}( X )  
&= \frac{P(1) L(1, \chi_{D})^{\eta}}{\Gamma(\eta)} X (\log X)^{\eta-1} \left(  1+ O_{\eta} \left( \frac{L_{N}^{2 e \eta +2}}{\sqrt{\log X}} \right)\right)
\end{split}
\end{equation}  
 uniformly for $N \ge 1$ and $X \ge \exp (C L_{N}^{2 e \eta +2} )$ where $ L_{N} = \log (\omega(N)+3).$
\end{prop}  
\noindent
\textbf{Proof of the \propref{MVEBQF}:} Following the argument of \cite[Lemma 4.1]{LAUL}, we obtained our result. Here, we give details of our proof. By Perron's formula and \eqref{DSLF}, we have 
\begin{equation}\label{PerronF}
\begin{split}
\mathcal{E}_{\eta}( X )  &=\sum_{n \le X \atop \gcd(n,N) =1  } \mu^{2}(n) \eta^{\omega(n)} r^{*}_{\mathcal{Q}}(n) =  \frac{1}{2 \pi i} \int_{b-iT}^{b+iT} \zeta^{\eta}(s) L^{\eta}(s, \chi_{D})P(s) \frac{X^{s}}{s} ds + O(R) \\
\end{split}
\end{equation}
where $b= 1+\frac{1}{\log X}$, $T \ge 3$ and 
\begin{equation}
\begin{split}
R  &= X \sum_{n=1}^{\infty}\frac{\eta^{\omega(n)} r^{*}_{\mathcal{Q}}(n) }{n^{b}(1+T |\log(X/n)|)}.
\end{split}
\end{equation}
  From the Weil bound ($r^{*}_{\mathcal{Q}}(n) \ll_{\epsilon} n^{\epsilon}$ for any $\epsilon>0$), we chose $\epsilon= 1/ 2\log X$ so that $r^{*}_{\mathcal{Q}}(n) \ll_{\epsilon} n^{1/ 2 \log X}$ and use it to get  (with $b_{0} = 1+\frac{1}{2\log X}$)
  \begin{equation}
\begin{split}
R  & \ll X \sum_{n=1}^{\infty}\frac{\eta^{\omega(n)} }{n^{b_{0}}(1+T |\log(X/n)|)}. 
\end{split}
\end{equation}
We break the above summation into  following two parts and obtain an estimate for $R$. \\
$(i)$ $n$ with $|\log (X/n)| \le T^{-1/2}$  \quad $(ii)$ $n$ with $|\log (X/n)| > T^{-1/2}$. \\
In the first case, the estimate for  the summation over $n$ with $|\log (X/n)| \le T^{-1/2}$ is obtained as follows.
 \begin{equation*}
\begin{split}
  &  X \sum_{|n-X| \le X T^{-1/2}}\frac{\eta^{\omega(n)} }{n^{b_{0}}(1+T |\log(X/n)|)}  \ll \sum_{|n-X| \le X T^{-1/2}}  \left(\frac{X}{n} \right)^{b_{0}}\frac{\eta^{\omega(n)} }{(1+T |\log(X/n)|)}\\
& \ll  \sum_{|n-X| \le X T^{-1/2}} \eta^{\omega(n)}  \ll  \left( \sum_{|n-X| \le X T^{-1/2}} \eta^{2\omega(n)} \right)^{\frac{1}{2}}  \left( \sum_{|n-X| \le X T^{-1/2}} \right)^{\frac{1}{2}} \ll  \frac{X (\log X)^{\frac{\eta^2-1}{2}}} {T^{1/4}}.
\end{split}
\end{equation*}
Which is obtained using the Cauchy-Schwarz inequality and argument as in \cite[Lemma 4.1]{LAUL}. In the second case, an estimate is obtained as follows.
 \begin{equation*}
\begin{split}
  &  X \sum_{|n-X| > X T^{-1/2}}\frac{\eta^{\omega(n)} }{n^{b_{0}}(1+T |\log(X/n)|)}  \ll  \frac{X}{T^{1/4}} \sum_{|n-X| > X T^{-1/2}} \frac{\eta^{\omega(n)} } {n^{b_{0}}} \ll \frac{X}{T^{1/4}} (\log X)^ {\eta-1}.\\
\end{split}
\end{equation*}
  Thus, we have  
   \begin{equation}
\begin{split}
R  & \ll X \sum_{n=1}^{\infty}\frac{\eta^{\omega(n)} }{n^{b_{0}}(1+T |\log(X/n)|)} \ll {\rm max} \left\{ \frac{X (\log X)^{\frac{\eta^2-1}{2}}} {T^{1/4}}, \frac{X (\log X)^{\frac{\eta-1}{2}}} {T^{1/2}} \right\}.
\end{split}
\end{equation}
To complete the proof, we need to evaluate the integral given in  \eqref{PerronF}. Let $c$ be a positive  constant and $ \sigma_{c}(T) := 1 - \frac{c}{ \log T}$.  Following the argument as in \cite{LAUL}, let $r= 1/2 \log X$ and assume that $1-r >  \sigma_{c}(T) $. The truncated Hankel contour $\Gamma$ is a positively oriented contour formed from the circle $|s-1|=r$ excluding the points $s=1-r$ and joining the half segment  $[ \sigma_{c}(T), 1-r]$ which is traced out twice with respective argument $+\pi$ and $-\pi$. We apply the Cauchy residue theorem to the closed path consisting of vertical line segment $[b-iT, b+iT]$ and $\mathcal{P}_{v}^{\pm}:= [ \sigma_{c/2}(T), \sigma_{c/2}(T) \pm iT]$, two horizontal line segment  $\mathcal{P}_{h}^{\pm}:= [ \sigma_{c/2}(T) \pm iT, b \pm iT]$ and the Hankel contour $\Gamma$.

For $\Re(s) \ge \sigma_{c}(T)$, we observe that 
\begin{equation}\label{PSBound}
\begin{split}
P(s)  & \ll \left| \prod_{p \mid N}\left( 1+ \frac{1}{p^{s}}\right)^{\eta}  \prod_{p \mid N}  \left( 1- \frac{\chi_{D}(p)}{p^{s}}\right)^{\eta} \right|  \ll   \prod_{p \mid N}\left( 1+ \frac{1}{p^{ \sigma_{c}(T)}}\right)^{2\eta} \\
& \ll \exp \{2 \eta \exp(2c L_{N}/\log T) \log L_{N} \} \ll L_{N}^{2e \eta}
\\
\end{split}
\end{equation}
using Merten's third theorem and $p_{n} \approx n \log n$ ( $p_{n}$ denotes the $n^{\rm th}$ prime), provided \linebreak  $T \ge \exp(2c L_{N}) $. From the above estimate and Cauchy's integral formula, we obtain that (for $\Re(s) \ge \sigma_{c/2}(T)$)
 \begin{equation}\label{DPSBound}
\begin{split}
P'(s) &= \frac{1}{2 \pi i} \int_{|z-\sigma_{c/2}(T)|= \frac{c}{ \log T}} \frac{P(z)}{(z-s)^{2}} dz   \ll \exp \{2 \eta \exp(2c L_{N}/\log T) \log L_{N} \} \log T \\
& \ll L_{N}^{2e \eta} \log T.
\end{split}
\end{equation}
From the bound $ \zeta(s) \ll \log T$ for $s \in \mathcal{P}_{h}^{\pm} \cup \mathcal{P}_{h}^{\pm} \cup \Gamma$, and \eqref{PSBound}, it is easy to have 
\begin{equation}
\begin{split}
 \frac{1}{2 \pi i} \int_{\mathcal{P}_{h}^{\pm} \cup \mathcal{P}_{h}^{\pm}} \zeta^{\eta}(s) L^{\eta}(s, \chi_{D})P(s)   \frac{X^{s}}{s}ds   \ll L_{N}^{2e \eta} \left(  \frac{X}{T} + X^{\sigma_{c/2}(T)}\right) (\log T)^{\eta-1}. \\
\end{split}
\end{equation}
Using the property of $\zeta(s)$ and bound for $P'(s)$  (given in \eqref{DPSBound}), we have (for $s \in \Gamma$)
\begin{equation}
\begin{split}
\frac{((s-1) \zeta(s))^{\eta}}{s} L^{\eta}(s, \chi_{D}) P(s) = L^{\eta}(1, \chi_{D}) P(1) + O_{\eta}( L^{\eta}(1, \chi_{D}) L_{N}^{2e \eta} \log T |s-1| ) 
\end{split}
\end{equation}
provided $T \ge \exp(2c L_{N}) $. This allows to get 
\begin{equation}
\begin{split}
& \int_{ \Gamma} \zeta^{\eta}(s) L^{\eta}(s, \chi_{D})P(s) \frac{X^{s}}{s} ds  \\
 & \quad  = L^{\eta}(1, \chi_{D}) P(1) \int_{ \Gamma} (s-1)^{-\eta} X^{s} ds   + O_{\eta}\left( L^{\eta}(1, \chi_{D}) L_{N}^{2e \eta} \log T \int_{ \Gamma} |s-1|^{1-\eta}  X^{s} ds \right).
\end{split}
\end{equation}
Using   \cite[Corollary II.0.18]{Tenenbaum}, we have
\begin{equation*}
\begin{split}
  \int_{ \Gamma} (s-1)^{-\eta} X^{s} ds &= X (\log X)^{\eta-1} \left\{ 1+ O_{\eta} (e^{-\frac{c \log X}{2 \log T}}) \right\} \\ \quad {\rm and } \quad  \int_{ \Gamma} |s-1|^{1-\eta}  X^{s} ds &\ll \int_{\sigma_{c/2}(T)}^{1-r} (1-y)^{1-\eta} X^{y} dy + X^{1+r} r^{2-\eta}  \ll X (\log X)^{\eta-2}.
\end{split}
\end{equation*}
Thus, we have 
\begin{equation}
\begin{split}
& \int_{ \Gamma} \zeta^{\eta}(s) L^{\eta}(s, \chi_{D})P(s) \frac{X^{s}}{s} ds \\
 & \quad = \frac{L^{\eta}(1, \chi_{D}) P(1)}{\Gamma(\eta)}  X (\log X)^{\eta-1}  \left\{ 1+ O_{\eta} \left( e^{-\frac{c \log X}{2 \log T}} \right) \right\} + O_{\eta}\left(  L^{\eta}(1, \chi_{D}) L_{N}^{2e \eta} X \log T  (\log X)^{\eta-2}  \right). 
\end{split}
\end{equation}
Applying Cauchy residue theorem and Combining all the estimates, we have 
\begin{equation*}
\begin{split}
 & \frac{1}{2 \pi i} \int_{b-iT}^{b+iT} \zeta^{\eta}(s) L^{\eta}(s, \chi_{D})P(s) \frac{X^{s}}{s} ds \\
& =\frac{L^{\eta}(1, \chi_{D}) P(1)}{\Gamma(\eta)}  X (\log X)^{\eta-1}  \left\{ 1+ O_{\eta} (e^{-\frac{c \log X}{2 \log T}}) \right\} + O_{\eta}\left(  L^{\eta}(1, \chi_{D}) L_{N}^{2e \eta} \log T X (\log X)^{\eta-2}  \right) \\
& \quad  + O_{\eta}( L_{N}^{2e \eta} \left(  \frac{X}{T} + X^{\sigma_{c/2}(T)}\right) (\log T)^{\eta-1}).
\end{split}
\end{equation*}
Thus, we have
\begin{equation}
\begin{split}
\mathcal{E}_{\eta}( X )  &=\sum_{n \le X \atop \gcd(n,N) =1  } \mu^{2}(n) \eta^{\omega(n)} r^{*}_{\mathcal{Q}}(n) =   \frac{L^{\eta}(1, \chi_{D}) P(1)}{\Gamma(\eta)}  X (\log X)^{\eta-1} + {\rm Err} \\
\end{split}
\end{equation}
where
\begin{equation*}
\begin{split}
 {\rm Err} 
& =   O_{\eta} \left(\frac{L^{\eta}(1, \chi_{D}) P(1)}{\Gamma(\eta)}  X (\log X)^{\eta-1} (e^{-\frac{c \log X}{2 \log T}}) \right) + O_{\eta}\left(  L^{\eta}(1, \chi_{D}) L_{N}^{2e \eta} \log T X (\log X)^{\eta-2}  \right) \\
& \quad  + O_{\eta} \left( L_{N}^{2e \eta} \left(  \frac{X}{T} + X^{\sigma_{c/2}(T)}\right) (\log T)^{\eta-1}\right) + O_{\eta} \left( {\rm max} \left\{ \frac{X (\log X)^{\frac{\eta^2-1}{2}}} {T^{1/4}}, \frac{X (\log X)^{\frac{\eta-1}{2}}} {T^{1/2}} \right\} \right).
\end{split}
\end{equation*}
Clearly, it is easy to see that 
$$
P(1) \gg L_{N}^{-2}
$$ 
using Merten's theorem. We choose  $T = \exp( {c_{1} \sqrt{\log X}})$ so that $T \ge \exp(2c L_{N}) $ is satisfied. Hence, for each $X \ge  \exp(2c L_{N}^{2})$, The error term satisfies following bound.
\begin{equation*}
\begin{split}
 {\rm Err} 
& \ll \frac{L^{\eta}(1, \chi_{D}) P(1)}{\Gamma(\eta)}  X (\log X)^{\eta-1}   \frac{ L_{N}^{2e \eta+2}}{\sqrt{\log X}}. 
\end{split}
\end{equation*}
Thus, we have the required result. This completes the proof. 

\smallskip
Let $Y$ be a positive real number. From Deligne's estimate, it is well-known that for each prime $p$,  $\lambda_{f}(p) = 2 \cos \theta_{p}, $ and more generally
$$
\lambda_{f}(p^{m}) = \frac{\sin ((m+1)\theta_{p})}{\sin~ \theta_{p}}, 
$$
for any positive integer $m.$  We define a  multiplicative function $h_{Y}$  supported on square-free positive integers. The function $h_{Y}$ is given by: 
\begin{equation}\label{Auxfun}
\begin{split}
& \qquad \quad h_{Y}(p) =
\begin{cases}
 \alpha \left( \frac{\log p}{\log Y} \right),  \qquad   {\rm if } \qquad    p \le Y \quad{\rm and } \quad  p \nmid N, \\
-2,  \qquad \qquad \quad  {\rm if } \qquad    p > Y \quad{\rm and } \quad p \nmid N,\\
0,  \qquad \qquad  \qquad  {\rm if } \qquad     p \mid N, 
 \end{cases} \\
\end{split}
\end{equation}
where  $\quad \alpha : [0, 1]  \rightarrow{[-2, 2]}$ is such that, $\alpha(0) = 2$ and $\alpha(t) = 2 \cos(\frac{\pi}{m+1})$ if $\frac{1}{m+1} < t \leq \frac{1}{m}$ for $m \in  \mathbf{N}$.
  We make use of the function $h_{Y}$  to obtain a lower bound for the sum $S(f, \mathcal{Q}; Y^{u} )$ for some $u\ge 1$,  following the idea of \cite{KLSW2010}.
  
  \begin{prop}\label{PropLower}
Let $U \ge 1$ be a real number,   $h_{Y}(n)$ (defined in \eqref{Auxfun}) and $\alpha(t)$ be as above with $\alpha_{0}>0$. Let $N \le X^{U}$ be a positive integer. Then, we have
\begin{equation*}
\begin{split}
 \sum_{n \le Y^{u} \atop \gcd(n,N) =1 }   h_{Y}(n) r^{*}_{\mathcal{Q}}(n) &= (\sigma(u) + o_{\alpha, U}(1)) \frac{P(1) L(1, \chi_{D})^{\alpha_{0}}}{\Gamma(\alpha_{0})} (\log Y^{u})^{\alpha_{0}-1} Y^{u} 
\end{split}
\end{equation*}
uniformly for $u \in \left[\frac{1}{U}, U \right]$ where $P(s)$ is given in \eqref{PSS} and 
\begin{equation*}
\begin{split}
\sigma(u)  &= (u)^{\alpha_{0}-1} + \sum_{j=1}^{\infty} \frac{(-1)^{j}}{j}  I_{j}(u)  \\
\end{split}
\end{equation*}
with
\begin{equation*}
\begin{split}
 I_{j}(u) &= \int_{\Delta_{j}} (u- t_{1}- t_{2}- \cdots- t_{j})^{\alpha_{0}-1} \prod_{i=1}^{j} (\alpha_{0}- \alpha(t))  \frac{dt_{1} dt_{2} \cdots dt_{j}}{ {t_{1}t_{2} \cdots t_{j}}} \\
\end{split}
\end{equation*}
and
\begin{equation*}
\begin{split}
  & \Delta_{j} = \{ ( t_{1}, t_{2}, \cdots, t_{j}) \in [0, \infty) \mid  |t_{1}+ t_{2}+ \cdots+t_{j}| \le u \}.
\end{split}
\end{equation*}
\end{prop}

\textbf{ Proof of \propref{PropLower}:} We follow the argument of \cite[Lemma 6]{KMatomaki} to give a brief proof of the  \propref{PropLower}. From inclusion-exclusion type identity, we have
\begin{equation*}
\begin{split}
 \sum_{n \le Y^{u} \atop \gcd(n,N) =1 }   h_{Y}(n) r^{*}_{\mathcal{Q}}(n) &=   \sum_{n \le Y^{u} \atop \gcd(n,N) =1 } \alpha_{0}^{\omega(n)} r^{*}_{\mathcal{Q}}(n) \\
 & \quad + \sum_{j=1}^{\infty} \frac{(-1)^{j}}{j} \!\!\!\!\!\!\!\!  \sum_{p_{1}p_{2} \cdots p_{j} \le Y^{u} \atop \gcd({p_{1}p_{2} \cdots p_{j}, N) =1 }}  \prod_{i=1}^{j} (\alpha_{0}- h_{Y}(p_{i})) r^{*}_{\mathcal{Q}}(p_{i})  \sum_{n \le \frac{Y^{u}}{p_{1}p_{2} \cdots p_{j}} \atop \gcd(n,N) =1 } \alpha_{0}^{\omega(n)} r^{*}_{\mathcal{Q}}(n).
\end{split}
\end{equation*}
From the Chebatorev density theorem, as the class number $h(D)=1$, the density of primes represented by ${\mathcal{Q}}$ of discriminant $D$ is $1/2$, and for such primes, we have  $ r^{*}_{\mathcal{Q}}(p)=2$ and $0$ otherwise. See \cite[Eq. 1.1 \& 1.2]{AZaman} On applying \propref{MVEBQF},  we have 
\begin{equation*}
\begin{split}
& \sum_{n \le Y^{u} \atop \gcd(n,N) =1 }   h_{Y}(n) r^{*}_{\mathcal{Q}}(n) \\
 &=   \frac{P(1) L(1, \chi_{D})^{\eta}}{\Gamma(\eta)} Y^{u}
\begin{cases} 
  (\log Y^{u})^{\alpha_{0}-1} \\
\displaystyle{+ \sum_{j=1}^{\infty} \frac{(-1)^{j}}{j}  \sum_{p_{1}p_{2} \cdots p_{j} \le Y^{u} \atop \gcd({p_{1}p_{2} \cdots p_{j}, N) =1 }}  \prod_{i=1}^{j} (\alpha_{0}- h_{Y}(p_{i})) r^{*}_{\mathcal{Q}}(p_{i}) \frac{(\log \frac{Y^{u}}{p_{1}p_{2} \cdots p_{j}})^{\alpha_{0}-1}}{ {p_{1}p_{2} \cdots p_{j}}}}
 \end{cases} \\
 &=   \frac{P(1) L(1, \chi_{D})^{\eta}}{\Gamma(\eta)} Y^{u}
\begin{cases} 
  (\log Y^{u})^{\alpha_{0}-1} \\
\displaystyle{+ \sum_{j=1}^{\infty} \frac{(-1)^{j}}{j}  \sum_{p_{1}p_{2} \cdots p_{j} \le Y^{u} \atop \gcd({p_{1}p_{2} \cdots p_{j}, N) =1 }}  \prod_{i=1}^{j} (\alpha_{0}- h_{Y}(p_{i})) \frac{(\log \frac{Y^{u}}{p_{1}p_{2} \cdots p_{j}})^{\alpha_{0}-1}}{ {p_{1}p_{2} \cdots p_{j}}}}.
 \end{cases}
\end{split}
\end{equation*}
Thus, we have the exactly the same expression as in  the proof of \cite[Lemma 6]{KMatomaki}. So, removing the coprimality condition from the expression and then using the prime number theorem, we have the required result.

\begin{rmk}
Let $\alpha$ be a step function given by;  $\alpha: [0, \infty) \rightarrow \mathbb{R}$ with $\alpha(t) = \alpha_{k}$ when $t \in [x_{k}, x_{k+1}]$ for each $k= 0,1,2, \cdots K$.  Here $K > 0$ and $[x_{k}, x_{k+1}]$ is partition of $[0, \infty)$ (in our case $x_{0}=0$, $x_{k} = \frac{1}{k}$,  $\alpha_{0} =2$ and $\alpha_{k} = 2 \cos(\frac{\pi}{k+1})$). In  \cite[Lemma 6]{KMatomaki}, it was mentioned that the function, $\sigma(u)$ is the unique solution of the integral equation 
$$
u \sigma(u) = \int_{0}^{u} \sigma(t) \alpha(u-t) dt
$$
with the initial condition $\sigma(u)= u^{\alpha_{0}-1}$ for $u \in (0, x_{1}]$. For the step function $\alpha$, it is shown that $\sigma(4/3)>0$ in \cite[Lemma 8]{KMatomaki}.
Hence,  for $u = u_{0} = 4/3$, we have
\begin{equation}\label{PBound}
\begin{split}
 \sum_{n \le Y^{u_{0}} \atop \gcd(n,N) =1 }   h_{Y}(n) r^{*}_{\mathcal{Q}}(n) &= (\sigma(u_{0}) + o_{\alpha, U}(1)) \frac{P(1) L(1, \chi_{D})^{\alpha_{0}}}{\Gamma(\alpha_{0})} (\log Y)^{\alpha_{0}-1} Y^{u_{0}} >0.
\end{split}
\end{equation}

\end{rmk}

\section{Proof of Results}
\noindent
\textbf{Proof of \thmref{PropUpper}:}
Let $1 \le Z \le \frac{X}{2}$. In order to obtain an upper bound for the sum $S(f, \mathcal{Q}; X )$ given in \eqref{Sum-PIBQF}, we introduce a smooth compactly supported function $w(x)$ satisfying: $w(x) =1$ for $x \in [2Z, X],$ $w(x) = 0$ for $x<Z$ and $x> X+Z,$ and $w^{(r)}(x) \ll_{r} Z^{-r}$ for all $r\ge 0.$ In general, for any arithmetical function $f(n), $ we have 
\begin{equation}\label{UpperEstimate}
\sum_{n \le X} f(n) = \sum_{n = 1}^{\infty} f(n)w(n) +  O\left(\sum_{n< 2Z} |f(n)| \right) + O\left(\sum_{X< n< X+ Z} |f(n)| \right). 
\end{equation}
Moreover, by Mellin's inverse transform, we have 
\begin{equation}
 \sum_{n = 1}^{\infty} f(n)w(n) = \frac{1}{2 \pi i} \int_{(b)} \tilde w(s) \left(\sum_{n\ge 1} \frac{f(n)}{n^{s}}\right)  ds, 
 \end{equation}
where $b$ is a positive real number larger than the  abscissa of absolute convergence of \linebreak $\displaystyle{\sum_{n\ge 1} \frac{f(n)}{n^{s}}}$ and the Mellin's transform $\tilde w(s)$ is given by following integral:
$$
\tilde w(s) = \int_{0}^{\infty} w(x) x^{s} \frac{dx}{x}.
$$
We observe that due to integration by parts, 
\begin{equation}\label{FourierW}
\tilde w(s) =  \frac{1}{s(s+1)\cdots(s+m-1)}\int_{0}^{\infty} w^{(m)}(x) x^{s+m-1} dx \ll \frac{Z}{X^{1-\sigma}}  \left(\frac{X}{|s|Z}\right)^{m},
\end{equation}
for any $m\ge 0,$ where $\sigma = \Re(s).$ For details, we refer to \cite[Section 3]{Jiang-Lu}.

\smallskip 
\noindent
From Equation \eqref{UpperEstimate} with $f(n) = \lambda_{f}(n)  r^{*}_{\mathcal{Q}}(n),$ we have 
\begin{equation*}
\begin{split}
\sideset{}{^{\flat }} \sum_{n \le X \atop \gcd(n,N)=1} \lambda_{f}(n)  r^{*}_{\mathcal{Q}}(n) & = \sideset{}{^{\flat }} \sum_{n \ge 1 \atop \gcd(n,N) =1} \lambda_{f}(n)  r^{*}_{\mathcal{Q}}(n)w(n)  \\
& \quad  +  O\left(\sideset{}{^{\flat }} \sum_{n < 2Z \atop \gcd(n,N) =1 } |\lambda_{f}(n)  r^{*}_{\mathcal{Q}}(n)| \right) + O\left(\sideset{}{^{\flat }} \sum_{X< n< X+ Z \atop \gcd(n,N) =1 } |\lambda_{f}(n)  r^{*}_{\mathcal{Q}}(n)| \right). 
\end{split}
\end{equation*}
Moreover, by Mellin's inverse transform, we have 
\begin{equation}\label{Perron-Type}
\sideset{}{^{\flat }} \sum_{n \ge 1 \atop \gcd(n,N) =1} \lambda_{f}(n)  r^{*}_{\mathcal{Q}}(n)w(n)  = \frac{1}{2 \pi i} \int_{(b)} \tilde w(s) L(f, \mathcal{Q}; s )  ds, 
 \end{equation}
where $b= 1+ \epsilon$ for some arbitrarily small $ \epsilon>0.$ Since   $ \lambda_{f}(n)  r^{*}_{\mathcal{Q}}(n)  \ll n^{\epsilon},$ for any $\epsilon>0,$ hence,
\begin{equation}\label{Errorbound}
\begin{split}
  O\left(\sideset{}{^{\flat }} \sum_{n < 2Z \atop \gcd(n,N) =1 } |\lambda_{f}(n)  r^{*}_{\mathcal{Q}}(n)| \right) + O\left(\sideset{}{^{\flat }} \sum_{X< n< X+ Z \atop \gcd(n,N) =1 } |\lambda_{f}(n)  r^{*}_{\mathcal{Q}}(n)| \right)  \ll Z^{1+\epsilon}.
\end{split}
\end{equation}
We know that $L(f, \mathcal{Q}; s ) $ is holomorphic for all $s \in \mathbb{C}$. We shift the line of integration from $\Re(s) = 1+\epsilon$ to $\Re(s) = \frac{1}{2}+\epsilon$ and apply Cauchy's residue theorem to get 
\begin{equation}\label{RHSInt}
\sideset{}{^{\flat }} \sum_{n \ge 1 \atop \gcd(n,N) =1} \lambda_{f}(n)  r^{*}_{\mathcal{Q}}(n)w(n)  = \frac{1}{2 \pi i} \int_{(1 + \epsilon)} \tilde w(s) L(f, \mathcal{Q}; s )  ds  =   \frac{1}{2 \pi i} \int_{(1/2+\epsilon)} \tilde w(s) L(f, \mathcal{Q}; s )  ds. 
 \end{equation}
 Since $\tilde w(s)  \ll \frac{Z}{X^{1-\sigma}}  \left(\frac{X}{|s|Z}\right)^{m}$
for any $m \ge 0,$ so  the contribution for the integral over  \linebreak $|s| \ge T = \frac{X^{1+\epsilon}}{Z}$ on the right hand side of \eqref{RHSInt} is negligibly small, i.e., $O(X^{-A})$ for any large $A>0$ if one chooses sufficiently large $m>0.$ Hence, we have 
\begin{equation*}
\begin{split}
  \frac{1}{2 \pi i} \int_{(1/2+\epsilon)} \tilde w(s) L(f, \mathcal{Q}; s )  ds  & =   \frac{1}{2 \pi i} \int_{1/2+\epsilon-iT}^{1/2+\epsilon+iT} \tilde w(s) L(f, \mathcal{Q}; s ) ds + O(X^{-A}) \\
  & \ll \int_{-T}^{T} |\tilde w(1/2+\epsilon+it)| |L(f, \mathcal{Q}; 1/2+\epsilon+it )| dt + O(X^{-A}) \\
  & \ll  \int_{0}^{T} \frac{X^{\frac{1}{2}+\epsilon}}{|\frac{1}{2}+\epsilon+it|} |L(f, \mathcal{Q}; 1/2+\epsilon+it )| dt + O(X^{-A}) \\
  & \ll \left(\int_{0}^{1} + \int_{1}^{T}\right) \frac{X^{\frac{1}{2}+\epsilon}}{|\frac{1}{2}+\epsilon+it|} |L(f, \mathcal{Q}; 1/2+\epsilon+it )| dt + O(X^{-A}), \\
  \end{split}
 \end{equation*}
where the estimate in the last lines is obtained by substituting the bound for $\tilde w(s)$ when $m=1.$  Now, we substitute the decomposition $L(f, \mathcal{Q}; s ) = L(s , f) L( s, f \otimes \chi_{D})G(s)$ from  \lemref{Decomp} and by absolute convergence of $G(s)$ for $\Re(s)> \frac{1}{2}$, we get 
\begin{equation*}
\begin{split}
 & \frac{1}{2 \pi i} \int_{(1/2+\epsilon)} \tilde w(s) L(f, \mathcal{Q}; s )  ds  \\
  & \ll \left(\int_{0}^{1} + \int_{1}^{T}\right) \frac{X^{\frac{1}{2}+\epsilon}}{|\frac{1}{2}+\epsilon+it|} |L(1/2+\epsilon+it  , f) L( 1/2+\epsilon+it , f \otimes \chi_{D})| dt + O(X^{-A}) \\
  & \ll (Nk^{2}|D|X)^{\frac{1}{2}+\epsilon} + X^{\frac{1}{2}+\epsilon} \int_{1}^{T} \frac{|L(1/2+\epsilon+it  , f) L( 1/2+\epsilon+it , f \otimes \chi_{D})|}{t} dt + O(X^{-A}). \\
  \end{split}
 \end{equation*}
 We apply the dyadic division method and then Cauchy-Schwartz inequality to get   
\begin{equation*}
\begin{split}
 & \frac{1}{2 \pi i} \int_{(1/2+\epsilon)} \tilde w(s) L(f, \mathcal{Q}; s )  ds  \\
 &   \ll 
   \begin{cases}
  (Nk^{2}|D|X)^{\frac{1}{2}+\epsilon}  + O(X^{-A})\\
  + ~ X^{\frac{1}{2}+\epsilon} \underset{1 \le T_{1} \le T}{\rm max} \!\! \left\{ \frac{1}{T_{1}} \!\!  \left( \int_{\frac{T_{1}}{2}}^{T_{1}}  |L(1/2+\epsilon+it  , f)|^{2} dt \right)^{\frac{1}{2}} \!\! \left(\int_{\frac{T_{1}}{2}}^{T_{1}} L( 1/2+\epsilon+it , f \otimes \chi_{D})|^{2} dt \right)^{\frac{1}{2}}\right\}  \\
  \end{cases} \\
 &  \ll (Nk^{2}|D|X)^{\frac{1}{2}+\epsilon} + X^{\frac{1}{2}+\epsilon} \underset{1 \le T_{1} \le T}{\rm max} \left(\frac{1}{T_{1}} (Nk^{2} T_{1}^{2})^{\frac{1}{4}+\epsilon} (Nk^{2}|D|^{2}T_{1}^{2})^{\frac{1}{4}+\epsilon}\right)+ O(X^{-A})\\
 &  \ll (Nk^{2}|D|X)^{\frac{1}{2}+\epsilon}.
  \end{split}
 \end{equation*}
 The last inequality is obtained by substituting $T= \frac{X^{1+\epsilon}}{Z}$ and choosing $Z= X^{\frac{1}{2}+\epsilon}$.  Thus
 \begin{equation}\label{upperbound}
\begin{split}
 & \frac{1}{2 \pi i} \int_{(1/2+\epsilon)} \tilde w(s) L(f, \mathcal{Q}; s )  ds   \ll (Nk^{2}|D|X)^{\frac{1}{2}+\epsilon}. \\
  \end{split}
 \end{equation}
From equations \eqref{Perron-Type}, \eqref{Errorbound}, \eqref{RHSInt} and \eqref{upperbound} and choosing $Z = X^{\frac{1}{2}+\epsilon}$,
we have the required result.

\smallskip
 Concerning about second result, one needs to obtain a lower estimate for the sum $S(f, \mathcal{Q}; Y^{u_{0}} )$ under the certain assumption. Below, we sketch the proof to get a lower estimate for the sum $S(f, \mathcal{Q}; Y^{u_{0}} )$.

\smallskip
\noindent
\textbf{A lower bound for the sum $S(f, \mathcal{Q}; Y^{u_{0}} )$:} We recall that multiplicative function $h_{Y}$, supported on square-free positive integers,  is given by: 
\begin{equation*}
\begin{split}
& \qquad \quad h_{Y}(p) =
\begin{cases}
 \alpha \left( \frac{\log p}{\log Y} \right),  \qquad   {\rm if } \qquad    p \le Y \quad{\rm and } \quad  p \nmid N, \\
-2,  \qquad \qquad \quad  {\rm if } \qquad    p > Y \quad{\rm and } \quad p \nmid N,\\
0,  \qquad \qquad  \qquad  {\rm if } \qquad     p \mid N, 
 \end{cases} \\
\end{split}
\end{equation*}
where$\quad \alpha : [0, 1]  \rightarrow{[-2, 2]}$ is such that, $\alpha(0) = 2$ and $\alpha(t) = 2 \cos(\frac{\pi}{m+1})$ if $\frac{1}{m+1} < t \leq \frac{1}{m}$ for $m \in  \mathbf{N}$.

Assume that $\lambda_{f}(n) \ge 0$ for every square-free positive integer $n \le Y,$ which are represented by $\mathcal{Q}(\underline{x}).$  
Then, Hecke relation together with the assumption  $\lambda_{f}(p) \ge 0$ for every prime $p \le Y^{1/m}$ implies that $\lambda_{f}(p) \ge 2 \cos\left(\frac{\pi}{m+1}\right).$ This shows that $\lambda_{f}(p) \ge h_{Y}(p),$ for all primes upto $Y.$ Multiplicativity of these arithmetical functions shows that $\lambda_{f}(n) \ge h_{Y}(n),$ for all $n \le Y.$ Similar to \cite{KLSW2010}, it is observed that for $X= Y^{u_{0}}$ with $u_{0}=4/3$, 
\begin{equation}\label{LBound}
\begin{split}
S(f, \mathcal{Q}; X) & = \sideset{}{^{\flat }} \sum_{n \le X \atop \gcd(n,N) =1 } \lambda_{f}(n) r^{*}_{\mathcal{Q}}(n) \ge \sideset{}{^{\flat }} \sum_{n \le X \atop \gcd(n,N) =1 } h_{Y}(n) r^{*}_{\mathcal{Q}}(n)>0.
\end{split}
\end{equation}
Precisely, to see \eqref{LBound}, let $g_{Q}$ be a multiplicative function defined by convolution identity $\lambda^{*}_{f}(n)  =  (g_{Q} * h^{*}_{Y})(n)$,
where $\lambda^{*}_{f}(n) = \lambda_{f}(n)r^{*}_{\mathcal{Q}}(n)$ and  $h^{*}_{y}(n) = h_{Y}(n)r^{*}_{\mathcal{Q}}(n). $ This gives that \linebreak $g_{Q}(p) =\lambda_{f}(p)r^{*}_{\mathcal{Q}}(p)- h_{Y}(p)r^{*}_{\mathcal{Q}}(p).$ If $p$ is not represented by $\mathcal{Q}(\underline{x}),$ then $g_{Q}(p) =0.$ If $p$ is represented by $\mathcal{Q}(\underline{x}),$ then $g_{Q}(p) \ge 0$ for  all $p$ not dividing $N$ and upto $Y$ by definition of  $ h_{Y}(p).$ For $p >Y,$ with $p \nmid N,$ we have $h_{Y}(p)r^{*}_{\mathcal{Q}}(p) = -4$ and $-2\le \lambda_{f}(p) \le 2.$ Hence, $g_{Q}(p) \ge 4a+4\ge 0$ as  $-1\le a \le 1.$ Here, we have written the value taken by $ \lambda_{f}(p)r^{*}_{\mathcal{Q}}(p)$ as $4a$. Thus, we have $g_{Q}(n) \ge 0$ for all square-free positive integers $n.$ Now, we see that 
\begin{equation*}
\begin{split}
S(f, \mathcal{Q}; X ) & = \sideset{}{^{\flat }} \sum_{n \le X \atop \gcd(n,N) =1 } \lambda_{f}(n) r^{*}_{\mathcal{Q}}(n) =  \sideset{}{^{\flat }} \sum_{n \le X \atop \gcd(n,N) =1 } \sum_{d \mid n} g_{Q}(d) h_{Y}(n/d) r^{*}_{\mathcal{Q}}(n/d)\\
& =  \sideset{}{^{\flat }} \sum_{d \le X \atop \gcd(d,N) =1 }  g_{Q}(d)  \sideset{}{^{\flat }} \sum_{m \le X/d \atop \gcd(m,N) =1 }   h_{Y}(m) r^{*}_{\mathcal{Q}}(m) \ge \sum_{n \le X \atop \gcd(n,N) =1 }   h_{Y}(n) r^{*}_{\mathcal{Q}}(n)
\end{split}
\end{equation*} 
as $g_{Q}(1) =1 $ and $g_{Q}(n )\ge 0$ for each square-free positive integer $n.$ Hence, a positive lower bound for the sum  $S(f, \mathcal{Q}; X)$ occurs and it is given by the sum  $\displaystyle{ \sum_{n \le X^{u_{0}} \atop \gcd(n,N) =1 }   h_{Y}(n) r^{*}_{\mathcal{Q}}(n)}$ for some $u_{0} < 1$ given in \propref{PropLower}.

\bigskip
\noindent 
\textbf{Proof of \thmref{ExtMatKLSW}:}\quad 
Let  $X = Y^{u}$ with $u >1$. 
By definition,  $n_{f, \mathcal{Q}}$ be the largest integer among all $n \in \mathbb{N}$ such that $\lambda_{f}(n)\ge 0,$ with $\gcd(n, N)=1$ and $n = \mathcal{Q}(\underline{x})$ for some $\underline{x} \in \mathbb{Z}^{2}$. Let us write  $n_{f, \mathcal{Q}}= Y.$  In other word, $\lambda_{f}(n)\ge 0,$ for each $n \le Y$ with $\gcd(n, N)=1$ and $n = \mathcal{Q}(\underline{x})$ for some  $\underline{x} \in \mathbb{Z}^{2}.$  A key  idea to estimate $Y$ is to compare lower and upper bound for the sum 
\begin{equation*}\label{Sum-PIBQF}
\begin{split}
S(f, \mathcal{Q}; Y^{u} ) &=  \sideset{}{^{\flat }}\sum_{n= \mathcal{Q}(\underline{x}) \le Y^{u} \atop \gcd(n,N) =1  } \lambda_{f}(n) = \sideset{}{^{\flat }} \sum_{n \le Y^{u} \atop \gcd(n,N) =1  } \lambda_{f}(n) r_{Q}(n).\\
\end{split}
\end{equation*}
From \propref{PropUpper} and \eqref{PBound}, we have  (for $u \le u_{0}$)

\begin{equation*}
\begin{split}
 0< \frac{P(1) L(1, \chi_{D})^{2}}{\Gamma(\alpha_{0})} Y^{u} \log Y^{u} \ll \sideset{}{^{\flat }} \sum_{n \le Y^{u} \atop \gcd(n,N) =1 } h_{Y}(n) r^{*}_{\mathcal{Q}}(n) \le S(f, \mathcal{Q}; Y^{u} )   \ll {(Nk^{2}|D|)}^{\frac{1}{2}+\epsilon} Y^{\frac{u}{2}+\epsilon}.
\end{split}
\end{equation*} 
This shows that 
\begin{equation*}
\begin{split}
 Y^{u/2+\epsilon}   \ll {(Nk^{2}|D|)}^{\frac{1}{2}+\epsilon}  L(1, \chi_{D})^{2}  \qquad \Longrightarrow  \qquad Y \ll {(Nk^{2})}^{\frac{1}{u}+\epsilon} |D|^{\frac{2}{u} +\epsilon}
 \end{split}
\end{equation*} 
Now, we substitute $u = u_{0} = 4/3$ (as $u_{0} =4/3$ such that $\sigma(u_{0})>0$) to get 
$$
Y \ll {(Nk^{2} |D|^{2})}^{\frac{3}{4} + \epsilon} .
$$

\noindent
This completes the proof.


\smallskip
\textbf{Data availability statements:} Data sharing is not applicable to this article as no datasets were generated or analyzed. 


\end{document}